\documentclass[12pt]{amsart}

\usepackage{amssymb,amscd,amsthm, verbatim,amsmath,color,fancyhdr, mathrsfs}
\usepackage{graphicx}
\usepackage{turnstile}

\usepackage[letterpaper, left=2.5cm, right=2.5cm, top=2.5cm,
bottom=2.5cm,dvips]{geometry}

\newtheorem{thm}{Theorem}[section]

\newtheorem{lem}[thm]{Lemma}
\newtheorem{rem}[thm]{Remark}
\newtheorem{cor}[thm]{Corollary}

\newtheorem{defn}[thm]{Definition}
\def \p{\partial}
\def \O{\Omega}
\def \R{{\mathbb R}}
\def \e{\varepsilon}
\def \l{\lambda}

\numberwithin{equation}{section}

\title[System of inhomogeneous wave inequalities]{Critical criteria of Fujita type for a system of inhomogeneous wave inequalities in exterior domains}

\author{Mohamed Jleli}
\address{Department of Mathematics, College of Science, King Saud University, Riyadh 11451, Saudi Arabia}
\email{jleli@ksu.edu.sa}

\author{Bessem Samet}
\address{Department of Mathematics, College of Science, King Saud University, Riyadh 11451, Saudi Arabia}
\email{bsamet@ksu.edu.sa}

\author{Dong Ye}
\address{Center for Partial Differential Equations, School of Mathematical Sciences and Shanghai Key Laboratory of PMMP, East China Normal
University, Shanghai 200241, China}
\address{IECL, UMR 7502, D\'epartement de Math\'ematiques, Universit\'e de Lorraine, 57073 Metz, France}
 \email{dye@math.ecnu.edu.cn, dong.ye@univ-lorraine.fr}



\begin{document}
\maketitle

\begin{abstract}
We consider blow-up results for a system of inhomogeneous wave inequalities in exterior domains. We will handle three type boundary conditions: Dirichlet type, Neumann type and mixed boundary conditions. We use a unified approach to show the optimal criteria of Fujita type for each case. Our study yields naturally optimal nonexistence results for the corresponding stationary wave system and equation. We provide many new results and close some open questions.
\end{abstract}

\section{Introduction}
This paper is concerned with the study of existence and nonexistence of global weak solutions to the system of wave inequalities
\begin{eqnarray}\label{P}
\Box u \geq |x|^a |v|^p, \;\; \Box v \geq |x|^b |u|^q \quad \mbox{in }\; (0,\infty)\times \Omega^c.
\end{eqnarray}
Here $\Box:=\partial_{tt}-\Delta$ is the wave operator, $\Omega^c$ denotes the complement of $\O$, with $\O$ a bounded smooth open set in $\mathbb{R}^N$ containing the origin and $N\geq 2$. Let $p,q>1$ and $a,b\geq -2$.

\medskip
We will study \eqref{P} under three types of boundary conditions: the Dirichlet type condition:
\begin{equation}\label{BC1}
(u(t,x),v(t,x))\succeq (f(x),g(x)), \quad \mbox{on }\; (0,\infty) \times \partial \Omega;
\end{equation}
the Neumann type condition:
\begin{equation}\label{BC2}
\left(\frac{\partial u}{\partial \nu}(t,x),\frac{\partial v}{\partial \nu}(t,x)\right) \succeq (f(x),g(x)), \quad \mbox{on }\; (0,\infty)\times \partial \Omega;
\end{equation}
and the mixed boundary condition:
\begin{equation}\label{BC3}
\left(u(t,x),\frac{\partial v}{\partial \nu}(t,x)\right)\succeq (f(x),g(x)),\quad \mbox{on }\; (0,\infty)\times \partial \Omega,
\end{equation}
where $f,g\in L^1(\partial \Omega, \R_+)$ are two fixed functions and $\nu$ is the outward unit normal vector on $\partial \Omega$, relative to $\Omega^c$. By the notation $\succeq$, we mean
the partial order on $\mathbb{R}^2$, that is
$$
(y_1,y_2)\succeq (z_1,z_2) \Longleftrightarrow  y_i\geq z_i,\,\,i=1,2.
$$
We write $y\succ z$, for $y,z\in \mathbb{R}^2$ if $y\succeq z$ and $y\ne z$.

\medskip
The large-time behavior of solutions to the wave equation
\begin{eqnarray}\label{PRN}
\Box u = |u|^p \quad\mbox{in}& [0,\infty)\times\mathbb{R}^N
\end{eqnarray}
has been studied extensively since four decades. Inspired by the seminal work of John \cite{J} in $\R^3$, Strauss conjectured in \cite{St} that for each $N \geq 2$, there exists a critical exponent $p_c(N)$ of Fujita type for the global existence question to \eqref{PRN} with compactly supported data, and it should be the positive root of the polynomial
\begin{align}
\label{PCN}
(N-1)p^2-(N+1)p-2=0.
\end{align}
This conjecture is finally showed to be true for all dimensions $N\geq 2$ after twenty-five years of efforts, see for instance \cite{J, G1, G2, SI, SC, GE, YZ, Z} and the references therein. More precisely, let $N \geq 2$ and
\begin{align*}
p_c(N) = \frac{N+1 + \sqrt{N^2 + 10N - 7}}{2(N-1)},
\end{align*}
then
\begin{itemize}
\item for any $(u, \p_tu)|_{t = 0}$ compactly supported with positive average, the solution to \eqref{PRN} blows-up in a finite time if $1 < p \leq p_c(N)$;
\item if $p > p_c(N)$, there are compactly supported initial conditions $(u, \p_t
u)|_{t = 0} \succ (0, 0)$ such that the solution to \eqref{PRN} exists globally in time.
\end{itemize}

\medskip
The wave inequality in the whole space was firstly studied by Kato \cite{K}:
\begin{align}
\label{PERN}
\Box u \geq |u|^p \;\; \mbox{ in } [0, \infty) \times \R^N.
\end{align}
He found another critical exponent $\widetilde p_c(N) = \frac{N+1}{N-1}$. Pohozaev \& Veron \cite{PV} generalized Kato's work and pointed out the sharpness of $\widetilde p_c$ for \eqref{PERN}. More precisely, they proved that,
\begin{itemize}
\item for any $N \geq 2$ and $1 < p \leq \widetilde p_c(N)$, there is no global weak solution to \eqref{PERN}, if
\begin{align}
\label{PV1}
\int_{\R^N} \p_t u(0, x) dx > 0;
\end{align}
\item inversely, if $p > \widetilde p_c(N)$, there are positive global solutions satisfying \eqref{PERN} and \eqref{PV1}.
\end{itemize}

\medskip
A natural question is to understand the wave equation or inequality on other unbounded domains of $\R^N$. The study of blow-up for wave equation on exterior domains was initialized by Zhang in \cite{Zhang}. Among many other things, he considered the inhomogeneous equation
\begin{eqnarray}\label{PRNE}
\Box u = |x|^\alpha |u|^p \quad\mbox{in}& (0,\infty)\times \Omega^c,
\end{eqnarray}
where $N \geq 3$, $\alpha > -2$ and $\Omega\subset \R^N$ is a smooth bounded set. Under the Neumann boundary condition $\frac{\p u}{\p\nu} = f \geq 0$ on $(0, \infty)\times \p\O$, Zhang showed that the critical exponent becomes now $\frac{N+\alpha}{N-2}$:
\begin{itemize}
\item when $1 < p < \frac{N+\alpha}{N-2}$, \eqref{PRNE} has no global solution if $f\not\equiv 0$;
\item when $p > \frac{N+\alpha}{N-2}$, problem \eqref{PRNE} has global solutions for some $f > 0$.
\end{itemize}
However, the Dirichlet boundary condition case was left open, see Remark 1.5 of \cite{Zhang}. Recently the special case with $\alpha = 0$ and $\O = B_r$ was studied in \cite{JS}. Here and after, $B_r$ denotes the ball centered at $0$ with radius $r > 0$. Our study for \eqref{P} will yield an optimal answer for \eqref{PRNE} under the Dirichlet boundary condition, see Corollary \ref{Cnew1} below.

\medskip
Here we are interested to understand the blow-up of solutions to \eqref{P} under various boundary conditions \eqref{BC1}, \eqref{BC2} and \eqref{BC3}. We will determine the critical criteria of Fujita type for $(p, q)$ in each case, without any assumption on the initial data. As far as we know, we are not aware of such results concerning system of wave equations or inequalities. The study for \eqref{P} yields natural consequences for the corresponding stationary system, which seem also to be new for the Neumann type condition and the mixed boundary condition, see Corollary \ref{Cnew} below. We are confident that our ideas can be adapted for other situations, as damped wave operators, parabolic operators or higher order operators.

\medskip
Before stating our results, let us mention in which sense the solutions are considered. Denote
$$
Q = (0, \infty) \times \Omega^c\quad\mbox{and}\quad \Gamma=(0,\infty)\times \partial \Omega.
$$
We introduce the test function space
$$
\mathcal{D}= \left\{\varphi \in C^2_{\rm cpt}(Q, \mathbb{R}_+):\,  \varphi|_{\Gamma} = 0, \frac{\partial \varphi}{\partial \nu}|_{\Gamma} \leq 0\right\}.
$$
Here, $C^2_{\rm cpt}(Q, \R_+)$ means the space of nonnegative $C^2$ functions compactly supported in $Q$. Notice that $\O^c$ is closed and $\Gamma \subset Q$.

\begin{defn}\label{Def1}
A pair $(u,v)\in L^q_{loc}(Q)\times L^p_{loc}(Q)$ is a global weak solution to \eqref{P}-\eqref{BC1}, if for any $\varphi \in \mathcal D$,
\begin{equation}\label{wks11}
\int_{Q} |x|^a |v|^p \varphi dx dt -\int_{\Gamma}\frac{\partial \varphi}{\partial \nu} f d\sigma dt\leq \int_{Q}u \Box \varphi dx dt
\end{equation}
and
\begin{equation}\label{wks12}
\int_{Q} |x|^b |u|^q \varphi dx dt -\int_{\Gamma}\frac{\partial \varphi}{\partial \nu} g d\sigma dt\leq \int_{Q}v \Box\varphi dxdt.
\end{equation}
\end{defn}

For Neumann boundary problem, we consider the test function space
$$
\mathcal N = \left\{\varphi \in C^2_{\rm cpt}(Q, \R_+):\,   \frac{\p \varphi}{\p \nu}|_{\Gamma} = 0\right\}.
$$

\begin{defn}\label{Def2}
A pair $(u,v)\in L^q_{loc}(Q)\times L^p_{loc}(Q)$
is called a global weak solution to \eqref{P}--\eqref{BC2}, if for any $\psi \in \mathcal N$,
\begin{equation}\label{wks21}
\int_{Q} |x|^a |v|^p \psi dx dt + \int_{\Gamma} \psi f d\sigma dt\leq \int_{Q}u \Box  \psi dx dt
\end{equation}
and
\begin{equation}\label{wks22}
\int_{Q} |x|^b |u|^q  \psi dx dt + \int_{\Gamma} \psi g d\sigma dt\leq \int_{Q}v \Box \psi dxdt.
\end{equation}
\end{defn}

For the mixed boundary problem, the natural test function space is then $\mathcal D \times \mathcal N$.
\begin{defn}\label{Def3}
A pair $(u,v)\in L^q_{loc}(Q)\times L^p_{loc}(Q)$
is a global weak solution to \eqref{P}--\eqref{BC3}, if for any $(\varphi, \psi) \in \mathcal D \times \mathcal N$, there holds \eqref{wks11} and \eqref{wks22}.
\end{defn}

Define
$$
I_f = \int_{\p\O} fd\sigma, \quad \mbox{for any }\; f \in L^1(\p\O).
$$
Let ${\rm sgn}$ denote the standard sign function over $\R$. Our main result is the following.
\begin{thm}\label{T1}
Assume that $(a,b)\succ (-2, -2)$, $f, g \in L^1(\p\O)$, $\left(I_f, I_g\right) \succ (0, 0)$ and $p, q > 1$. Let either $N= 2$; or $N \geq 3$ and
\begin{align}
\label{newc1}
\max\left\{{\rm sgn}(I_f)\times\frac{2p(q+1)+pb+a}{pq-1},\;\; {\rm sgn}(I_g)\times \frac{2q(p+1)+qa+b}{pq-1}\right\}> N.
\end{align}
Then
\begin{itemize}
\item[(i)] there exists no global weak solution to \eqref{P}--\eqref{BC1} if $f, g \geq 0$;
\item[(ii)] there exists no global weak solution to \eqref{P}--\eqref{BC2};
\item[(iii)] there exists no global weak solution to \eqref{P}--\eqref{BC3} if $p>2$ and $f \geq 0$.
\end{itemize}
Furthermore, if $\O = B_r$, the sign condition for $f, g$ can be erased in ${\rm (i)}$ and ${\rm (iii)}$.
\end{thm}

\begin{rem}
The condition \eqref{newc1} is equivalent to
$$ I_f > 0 \;\mbox{ and }\; \delta > N - 2;\quad \mbox{or}\quad I_g > 0 \;\mbox{ and }\;\gamma > N-2,
$$
where
\begin{align}
\label{dg}
\delta = \frac{a+2+p(b+2)}{pq-1}, \quad \gamma = \frac{b+2+q(a+2)}{pq-1}.
\end{align}
Therefore, \eqref{newc1} always holds true when $N = 2$, $p, q > 1$ and $(a, b) \succ ({-2}, {-2})$.
\end{rem}

In fact, the constants $\delta$, $\gamma$ come from the scaling transform of the stationary problem
\begin{align}
\label{SPnew1}
-\Delta u = |x|^av^p, \;\; -\Delta v = |x|^bu^q.
\end{align}
Let $(u, v)$ be a solution to the system \eqref{SPnew1}, then for any $\lambda > 0$, $u_\lambda(x) = \lambda^\delta u(\l x), v_\l(x) = \l^\gamma v(\l x)$ satisfy still \eqref{SPnew1}.

\begin{rem}
Assume that $N \geq 3$, $p, q > 0$, $pq > 1$, and
\begin{align}
\label{newc2}
0< \min(\delta, \gamma) \leq \max(\delta, \gamma) < N-2.
\end{align}
Let $(u_*, v_*)(x) = (A_u|x|^{-\delta}, A_v|x|^{-\gamma})$ with $A_u, A_v > 0$ given by
$$A_u^{pq-1} = \delta(N-2-\delta)[\gamma(N-2-\gamma)]^p, \quad A_v^{pq-1} = \gamma(N-2-\gamma)[\delta(N-2-\delta)]^q.$$
We can check that $(u_*,v_*)$ is a positive solution to \eqref{SPnew1} in $\R^N\backslash\{0\}$. If $\Omega$ is star-shaped with respect to the origin, there holds $\frac{\p u_*}{\p\nu}, \frac{\p v_*}{\p \nu} \geq 0$ on $\p\Omega$ with respect to $\Omega^c$. So $(u_*, v_*)$ is a stationary solution to \eqref{P} and satisfies all the boundary conditions \eqref{BC1}, \eqref{BC2} and \eqref{BC3} for suitable $f, g \geq 0$. This means that the condition \eqref{newc1} is optimal for the nonexistence of global solution to the wave system \eqref{P}.
\end{rem}

\begin{rem}
\label{signfg}
Assume that $B_{r_1} \subset \O$ with $r_1 > 0$. Let $a, b \leq 0$, $p, q > 0$, $pq > 1$. Similarly as above, there are suitable $A_1, A_2 > 0$ such that
$$u(t,x) = A_1(t+1)^{-\frac{2(p+1)}{pq - 1}}, \quad v(t,x) = A_2(t+1)^{-\frac{2(q+1)}{pq - 1}}$$
satisfy $\Box u = r_1^a v^p$ and $\Box v = r_1^b u^q$ in $\R_+ \times \R^N$. Therefore, $(u, v)$ resolves \eqref{P} and satisfies all the boundary conditions \eqref{BC1}, \eqref{BC2} and \eqref{BC3} with $f = g =0$. This means the necessity of the assumption $(I_f, I_g)\succ (0, 0)$ in Theorem \ref{T1} when $a, b \leq 0$.
\end{rem}

Clearly, Theorems \ref{T1} yields nonexistence results for the corresponding stationary problem
\begin{align}
\label{StaP}
-\Delta u \geq |x|^a |v|^p, \;\; -\Delta v \geq |x|^b|u|^q \quad \mbox{in }\; \Omega^c.
\end{align}

\begin{cor}
\label{Cnew}
Let $N \geq 2$, $f, g \in L^1(\p\O)$ and $(a,b)\succ (-2, -2)$. Assume that $\left(I_f, I_g\right) \succ (0, 0)$ and $p, q > 1$ satisfy \eqref{newc1}. Then \eqref{StaP} has no weak solution if one of the following conditions holds true:
\begin{itemize}
\item[(i)] $f, g \in L^1(\p\O, \R_+)$, $(u, v) \succeq (f, g)$ on $\p \Omega$;
\item[(ii)] $\left(\frac{\p u}{\p \nu}, \frac{\p v}{\p \nu}\right) \succeq (f, g)$ on $\p \Omega$;
\item[(iii)] $f \in L^1(\p\O, \R_+)$, $p > 2$ and $\left(u,\frac{\p v}{\p \nu}\right) \succeq (f, g)$ on $\p \Omega$.
\end{itemize}
\end{cor}
We refind Corollary 1.3 in \cite{Sun} for the Dirichlet boundary condition case, where $a, b > -2$ was assumed. It seems to be the first time that such nonexistence results are showed for \eqref{StaP} under the Neumann type condition or the mixed boundary condition. Similarly, the sign condition for $f, g$ can be erased if $\O = B_r$.

\medskip
Theorems \ref{T1} yields also new result for the following wave inequality in exterior domain
\begin{align}
\label{PSED}
\Box u \geq |x|^a |u|^p  \; \mbox{ in }\;  (0,\infty)\times \Omega^c, \quad u(t,x) \geq f(x) \;\mbox{ on } \; (0,\infty)\times \partial \Omega,
\end{align}
and answers an open question proposed in Remark 1.5 of \cite{Zhang}.

\begin{cor}
\label{Cnew1}
Let $a>-2$, $f\in L^1(\partial \Omega, \R_+)$ and $N\geq 3$. If
\begin{equation}\label{psat}
I_f > 0 \quad \mbox{and}\quad 1<p<\frac{N+a}{N-2},
\end{equation}
there is no global weak solution in $L^p_{loc}(Q)$ to \eqref{PSED}. In other words, $p^*=\frac{N+a}{N-2}$
is the Fujita critical exponent for \eqref{PSED} if $N\geq 3$, $a > -2$.
\end{cor}
Indeed, when $b=a$, $q=p > 1$
$$
N< \frac{2p(q+1)+pb+a}{pq-1} = \frac{2p+a}{p-1} \Longleftrightarrow p<\frac{N+a}{N-2}.
$$
Taking $(v, b, q, g) = (u, a, p, f)$  in \eqref{P}--\eqref{BC1}, we deduce the above nonexistence result from part (i) of Theorem \ref{T1}. Again the condition $f \geq 0$ is not necessary if $\O = B_r$. On the other hand, \eqref{PSED} admits positive solution for $a > -2$, $p>\frac{N+a}{N-2}$, $N \geq 3$ and  $f>0$ with $\|f\|_\infty$ is sufficiently small (see \cite[Proposition 6.1]{Zhang}).

\begin{rem}
Similarly, for the exterior Neumann inequality
\begin{align*}
\Box u \geq |x|^a |u|^p  \; \mbox{ in }\;  (0,\infty)\times \Omega^c, \quad \frac{\p u}{\p\nu}(t,x) \geq f(x) \;\mbox{ on } \; (0,\infty)\times \partial \Omega,
\end{align*}
we refind the critical exponent $p^*=\frac{N+a}{N-2}$ as indicated by \cite[Theorem 1.4]{Zhang}.
\end{rem}

Let us say some words for our approach which is based on suitable test functions and integral estimates. At first glance it looks like the method in \cite{Zhang, Sun} or similar works for the blow-up study in exterior domains, however some key choices are completely different.
\begin{itemize}
\item In most previous works, we use cut-off functions with fixed scaling for the time variable $t$, we obtain then integral estimates on cylinder type domain $Q_D = \Sigma_t\times \Sigma_x$ where $|\Sigma_x| \sim R^N$ and $\Sigma_t$ is of length $CR$ or $CR^2$. Here we consider a large scale for $t$ by choosing $|\Sigma_t| \sim R^\theta$ with $\theta$ large enough.
\item In \cite{Zhang, Sun}, they often use test functions with support away from the boundary $\p\O$, hence it's more difficult to observe the effect of the Dirichlet boundary condition. In this work, we make use of harmonic function on $\O^c$ with zero boundary condition, which permits to cut off only at infinity.
\end{itemize}
These ideas make our method more transparent, for example we avoid the iterative step used in \cite{Zhang, Sun}.

\medskip
The paper is organized as follows. In section 2, we establish some preliminary estimates that will be used in the proof of our main results.
In Section \ref{sec3}, we prove Theorem \ref{T1} in two dimensional case.
The proof of Theorem \ref{T1} for $N\geq 3$ is given in Section \ref{sec4}. Finally, some open questions are raised in Section \ref{sec5}.

\medskip
The symbols $C$ or $C_i$ denote always generic positive constants, which are independent of the scaling parameter $T$ and the solutions $u, v$. Their values could be changed from one line to another. We will write $B:= B_1$ for the unit ball, and we will use the notation $h \sim k$ for two positive functions or quantities, which satisfy $C_1 h \leq k \leq C_2 h$.

\section{Preliminary estimates}\label{sec2}
Let $N\geq 2$. We introduce the following harmonic function in $\O^c$:
\begin{align*}
-\Delta H_\Omega = 0  \;\; \mbox{in }\; \Omega^c, \quad H_\Omega = 0 \;\; \mbox{on }\; \p\Omega;
\end{align*}
and
\begin{align*}
\lim_{|x|\to\infty}\frac{H_\Omega(x)}{\ln|x|} = 1  \;\; \mbox{if }\; N = 2; \quad \lim_{|x|\to\infty} H_\Omega(x) = 1  \;\; \mbox{if }\; N \geq 3.
\end{align*}
Clearly $H_\Omega$ is uniquely determined and $H_\Omega > 0$ in $\overline\Omega^c$.

\medskip
We need also two cut-off functions. Let $\xi\in C^\infty(\R^N)$ satisfies
$$
0\leq \xi\leq 1;\quad \xi \equiv 1 \;\mbox{ in } B; \quad \xi(x) \equiv 0 \;\mbox{ if }\; |x|\geq 2.
$$
Fix also $\vartheta \in C^\infty(\R)$ such that
$$
\vartheta \geq 0, \quad \vartheta\not\equiv 0,\quad {\rm supp}(\vartheta) \subset (0, 1).
$$
For $0<T<\infty$, let
$$
\Xi_T(x) = H_\O(x)\xi\left(\frac{x}{T}\right)^k \quad \mbox{in }\; \O^c
$$
and
$$
\vartheta_T(t)=\vartheta\left(\frac{t}{T^\theta}\right)^k \quad \mbox{in }\; (0,\infty).
$$
Here, $k \geq 2$ and $\theta >0$ are constants to be chosen later.

\medskip
Consider
$$
D_T(t,x)=\vartheta_T(t)\Xi_T(x),\quad (t,x)\in (0,\infty)\times \O^c
$$
and
$$
N_T(x)=\vartheta_T(t) \xi\left(\frac{x}{T}\right)^k,\quad (t,x)\in (0,\infty)\times \O^c.
$$
Obviously, for any $T > {\rm dist}(0, \p\O)$ and $\theta>0$,
$$
(D_T,N_T)\in \mathcal{D}\times\mathcal{N}.
$$
Denote $H := H_B$, i.e.
\begin{eqnarray*}
H(x)=\left\{\begin{array}{lll}
\ln|x| &\mbox{if }\; N=2,\\
1-|x|^{2-N} &\mbox{if }\; N\geq 3.
\end{array}
\right.
\end{eqnarray*}

In the following, we will give some integral estimates for $D_T$ and $N_T$. Our approach uses only the asymptotic behavior of $H_\Omega$ and its derivatives at infinity. For simplicity, we will detail our proof only for the unit open ball $B$. The readers can be convinced easily that the same ideas work well for general smooth open sets $\Omega$. More precisely, as $B_{r_2} \subset \Omega \subset B_{r_1}$ with $r_1 > r_2 > 0$, thanks to the maximum principle, we have $H_{B_{r_1}} \leq H_\Omega \leq H_{B_{r_2}}$ in $B_{r_1}^c$. The standard elliptic theory yields that $|\nabla^k H_\Omega|(x) \sim |\nabla^k H|(x)$ as $|x|\to \infty$, for all $k \geq 0$.

\medskip
The following estimates follow from standard calculations.
\begin{lem}\label{LL1}
Let $N=2$, $\alpha\in \mathbb{R}$ and $\beta>-1$. There holds, as $T \to +\infty$,
$$
\int_{1<|x|< T} |x|^\alpha (\ln|x|)^\beta\,dx \sim
\left\{
\begin{array}{lll}
1 &\mbox{if }\; \alpha< -2,\\
(\ln T)^{\beta+1} &\mbox{if }\; \alpha=-2,\\
T^{\alpha+2}(\ln T)^\beta &\mbox{if }\;\alpha> -2;
\end{array}
\right.
$$
and for any $\alpha,\beta\in \mathbb{R}$, we have
$$
\int_{T<|x|<2 T} |x|^\alpha (\ln|x|)^\beta\,dx \sim T^{\alpha+2}(\ln T)^\beta,\quad \mbox{as }T\to +\infty.
$$
\end{lem}

\begin{lem}\label{LL3}
Let $N\geq3$, $\alpha\in \mathbb{R}$ and $\beta>-1$. There holds, as $T\to +\infty$.
$$
\int_{1<|x|< T} |x|^\alpha \left(1-|x|^{2-N}\right)^\beta dx \sim
\left\{
\begin{array}{lll}
1 &\mbox{if }\; \alpha< -N;\\
\ln T &\mbox{if }\; \alpha=-N,\\
T^{\alpha+N} &\mbox{if }\; \alpha> -N;
\end{array}
\right.
$$
and for any $\alpha,\beta\in \mathbb{R}$, we have
$$
\int_{T<|x|< 2T} |x|^\alpha \left(1-|x|^{2-N}\right)^\beta dx \sim
T^{\alpha+N},\quad \mbox{as }T\to +\infty.
$$
\end{lem}

\subsection{Estimates involving $D_T$}
By the definitions of $\Xi_T$ and $\vartheta_T$, there holds
$$
\Delta \Xi_T(x) = \frac{2k}{T}\xi\left(\frac{x}{T}\right)^{k-1}\nabla H_\O(x)\nabla\xi\left(\frac{x}{T}\right) + \frac{H_\O(x)}{T^2}\left[k(k-1)\xi^{k-2}\left|\nabla\xi\right|^2 + k\xi^{k-1}\Delta\xi\right]\left(\frac{x}{T}\right)
$$
and
$$
\vartheta_T''(t)=\frac{1}{T^{2\theta}} \vartheta\left(\frac{t}{T^\theta}\right)^{k-2} \left[k(k-1)\vartheta'^2+k\vartheta \vartheta''\right]\left(\frac{t}{T^\theta}\right).
$$
We deduce then
\begin{lem}\label{LL10}
$$
|\Delta \Xi_T(x)| \leq C_k
\left(\frac{H_\O(x)}{T^2}+ \frac{|x|^{1-N}}{T}\right)\xi\left(\frac{x}{T}\right)^{k-2}\chi_{\{T < |x|< 2T\}} \quad \mbox{in } \O^c.
$$
and
$$
|\vartheta_T''(t)|\leq \frac{C_k}{T^{2\theta}}  \vartheta\left(\frac{t}{T^\theta}\right)^{k-2} \chi_{\{0< t< T^\theta\}} \quad \mbox{in } (0,\infty).
$$
\end{lem}

As the harmonic function $H_\O$ has very different behaviors in dimension two comparing to higher dimensions, we separate the study in two cases: $N = 2$ and $N \geq 3$.
\begin{lem}\label{LL11}
Let $\tau\in \mathbb{R}$, $\theta > 0$, $m>1$, $k>\frac{2m}{m-1}$ and $N=2$. We have, as $T\to +\infty$,
$$
\int_{Q} |x|^{\frac{-\tau}{m-1}} D_T^{\frac{-1}{m-1}} |\partial_{tt} D_T|^{\frac{m}{m-1}} dxdt = \left\{
\begin{array}{lll}
O\left(T^{2-\frac{\tau + (m+1)\theta}{m-1}}\ln T\right) &\mbox{if }\; \tau<2(m-1),\\
O\left(T^{-\frac{(m+1)\theta}{m-1}}(\ln T)^{2}\right) &\mbox{if }\; \tau=2(m-1),\\
O\left(T^{-\frac{(m+1)\theta}{m-1}}\right) &\mbox{if }\; \tau> 2(m-1).
\end{array}
\right.
$$
\end{lem}

\begin{proof}
Without loss of generality, let $\O = B$. By the definition of $D_T$ and Lemma \ref{LL10}, we get
\begin{align*}
\begin{split}
\int_{Q} |x|^{\frac{-\tau}{m-1}} D_T^{\frac{-1}{m-1}} |\partial_{tt} D_T|^{\frac{m}{m-1}} dxdt &=
\int_0^\infty \vartheta_T(t)^{\frac{-1}{m-1}} |\vartheta_T''(t)|^{\frac{m}{m-1}} dt \times\int_{B^c} |x|^{\frac{-\tau}{m-1}} \Xi_T(x) dx\\
&\leq C T^{\frac{-2\theta m}{m-1}} \int_0^{T^\theta} \vartheta^{k-\frac{2m}{m-1}}\left(\frac{t}{T^\theta}\right) dt \times \int_{1<|x|<2T}  |x|^{\frac{-\tau}{m-1}} H(x) dx\\
& \leq CT^{-\frac{(m+1)\theta}{m-1}} \int_{1<|x|< 2T} |x|^{\frac{-\tau}{m-1}} H(x) dx.
\end{split}
\end{align*}
Using Lemma \ref{LL1} with $\alpha=\frac{-\tau}{m-1}$ and $\beta=1$, we obtain the claimed estimate.
\end{proof}

Similarly, we deduce from Lemma \ref{LL3} that
\begin{lem}\label{LL12}
Let $\tau\in \mathbb{R}$, $\theta > 0$, $m>1$, $k>\frac{2m}{m-1}$  and $N\geq 3$. There holds, as $T\to +\infty$,
$$
\int_{Q} |x|^{\frac{-\tau}{m-1}} D_T^{\frac{-1}{m-1}} |\partial_{tt} D_T|^{\frac{m}{m-1}} dxdt = \left\{
\begin{array}{lll}
O\left(T^{N-\frac{\tau + (m+1)\theta}{m-1}} \right)&\mbox{if }\; \tau< N(m-1),\\
O\left(T^{-\frac{(m+1)\theta}{m-1}}\ln T \right)&\mbox{if }\; \tau=N(m-1),\\
O\left(T^{-\frac{(m+1)\theta}{m-1}}\right) &\mbox{if }\; \tau> N(m-1).
\end{array}
\right.
$$
\end{lem}

Furthermore, there holds
\begin{lem}\label{LL18}
Let $\tau  \in \R$, $\theta>0$, $m>1$, $k >\frac{2m}{m-1}$ and $N= 2$. Then
$$
\int_{Q} |x|^{\frac{-\tau}{m-1}} D_T^{\frac{-1}{m-1}} |\Delta D_T|^{\frac{m}{m-1}} dxdt=O\left(T^{\theta -\frac{\tau+2}{m-1}}\ln T\right),\quad \mbox{as }T\to +\infty.
$$
\end{lem}

\begin{proof}
Consider still $\O = B$. By the definition of $D_T$,
\begin{align}\label{SP1}
\begin{split}
&\int_{Q} |x|^{\frac{-\tau}{m-1}} D_T^{\frac{-1}{m-1}} |\Delta D_T|^{\frac{m}{m-1}} dxdt\\
= & \; \int_0^\infty \vartheta_T(t)\,dt \times
\int_{B^c} |x|^{\frac{-\tau}{m-1}} H(x)^{\frac{-1}{m-1}} \xi\left(\frac{x}{T}\right)^{\frac{-k}{m-1}}
|\Delta \Xi_T(x)|^{\frac{m}{m-1}} dx\\
= & \; CT^\theta\int_{B^c} |x|^{\frac{-\tau}{m-1}} H(x)^{\frac{-1}{m-1}} \xi\left(\frac{x}{T}\right)^{\frac{-k}{m-1}}
|\Delta \Xi_T(x)|^{\frac{m}{m-1}} dx.
\end{split}
\end{align}
Applying Lemma \ref{LL10}, there holds, for any $|x| > 1$,
\begin{align}\label{SP2}
\begin{split}
& H(x)^{\frac{-1}{m-1}} \xi\left(\frac{x}{T}\right)^{\frac{-1}{m-1}}|\Delta \Xi_T(x)|^{\frac{m}{m-1}}\\
\leq & \; CH(x)^{\frac{-1}{m-1}} \xi^{k-\frac{2m}{m-1}}\left(\frac{x}{T}\right)\left(\frac{H(x)}{T^2}+ \frac{|x|^{1-N}}{T}\right)^\frac{m}{m-1}\chi_{\{T < |x|< 2T\}}\\
\leq & \; C\Big[T^{-\frac{2m}{m-1}}H(x) + T^{-\frac{m}{m-1}}H(x)^{\frac{-1}{m-1}}|x|^\frac{(1-N)m}{m-1}\Big] \chi_{\{T < |x|< 2T\}}.
\end{split}
\end{align}
Combining \eqref{SP1}--\eqref{SP2}, we obtain
\begin{align*}
\begin{split}
&\int_{Q} |x|^{\frac{-\tau}{m-1}} D_T^{\frac{-1}{m-1}} |\Delta D_T|^{\frac{m}{m-1}} dxdt\\
\leq & \; CT^{\theta -\frac{2m}{m-1}}\int_{T < |x| < 2T} |x|^{\frac{-\tau}{m-1}} H(x) dx + CT^{\theta-\frac{m}{m-1}}\int_{T < |x| < 2T}H(x)^{\frac{-1}{m-1}}|x|^\frac{-\tau + (1-N)m}{m-1} dx\\
= & \; O\left(T^{\theta-\frac{\tau+2}{m-1}}\ln T\right),
\end{split}
\end{align*}
as $T$ goes to $+\infty$. The last line is given by $N = 2$ and Lemma \ref{LL1}.
\end{proof}

Very similarly, using the expression of $H$ and Lemma \ref{LL3}, we have
\begin{lem}\label{LL19}
Let $\tau \in \R$, $\theta>0$, $m>1$, $k> \frac{2m}{m-1}$ and $N\geq 3$, then
$$
\int_{Q} |x|^{\frac{-\tau}{m-1}} D_T^{\frac{-1}{m-1}} |\Delta D_T|^{\frac{m}{m-1}} dxdt=O\left(T^{N-2 + \theta - \frac{\tau+2}{m-1}}\right),\quad \mbox{as }T\to +\infty.
$$
\end{lem}

\subsection{Estimates involving $N_T$}

\begin{lem}\label{LL13}
Let $\tau  \in \R$, $\theta>0$, $m>1$, $k>\frac{2m}{m-1}$ and
$N\geq 2$. There holds, as $T\to +\infty$,
$$
\int_{Q} |x|^{\frac{-\tau}{m-1}} N_T^{\frac{-1}{m-1}} |\partial_{tt} N_T|^{\frac{m}{m-1}} dxdt =\left\{
\begin{array}{lll}
O\left(T^{-\frac{(m+1)\theta}{m-1}}\ln T\right) &\mbox{if }\; \tau \geq N(m-1),\\
O\left(T^{N -\frac{\tau + (m+1)\theta}{m-1}}\right) &\mbox{if }\; \tau < N(m-1).
\end{array}
\right.
$$
\end{lem}

\begin{proof}
Consider $\O = B$. By the definition of $N_T$ and Lemma \ref{LL10}, we get
\begin{align*}
\begin{split}
\int_{Q} |x|^{\frac{-\tau}{m-1}} N_T^{\frac{-1}{m-1}} |\partial_{tt} N_T|^{\frac{m}{m-1}} dxdt &= \int_0^\infty \vartheta_T(t)^{\frac{-1}{m-1}} |\vartheta_T''(t)|^{\frac{m}{m-1}} dt \times \int_{B^c} |x|^{\frac{-\tau}{m-1}} \xi^k \left(\frac{x}{T}\right) dx\\
& \leq CT^{-\frac{(m+1)\theta}{m-1}} \int_{1 < |x| < 2T} |x|^{\frac{-\tau}{m-1}} dx,
\end{split}
\end{align*}
The desired estimate follows directly from Lemmas \ref{LL1}--\ref{LL3} with $\alpha=\frac{-\tau}{m-1}$ and $\beta=0$.
\end{proof}

\begin{lem}\label{LL20}
Let $\tau  \in \R$, $\theta>0$, $m>1$, $k>\frac{2m}{m-1}$ and $N\geq 2$. Then
$$
\int_{Q} |x|^{\frac{-\tau}{m-1}} N_T^{\frac{-1}{m-1}} |\Delta N_T|^{\frac{m}{m-1}}\,dx\,dt=O\left(T^{N- 2 + \theta - \frac{\tau+2}{m-1}}\right),\quad \mbox{as }T\to +\infty.
$$
\end{lem}

\begin{proof}
As in Lemma \ref{LL10}, there holds
\begin{align}
\label{newExi}
\left|\Delta \left[\xi^k\left(\frac{x}{T}\right)\right]\right| \leq C_kT^{-2}\xi\left(\frac{x}{T}\right)^{k-2}\chi_{\{T < |x|< 2T\}}.
\end{align}
We can claim the mentioned estimate similarly as for Lemmas \ref{LL18}. \end{proof}

\subsection{Estimates involving $D_T$ and $N_T$}
The following are some estimates necessary to handle the mixed boundary problem \eqref{P}--\eqref{BC3}.

\begin{lem}\label{LL16}
Let $\tau  \in \R$, $\theta>0$, $m>2$, $k>\frac{2m}{m-1}$
  and $N \geq 2$. There holds, as $T \to +\infty$,
$$
\int_{Q} |x|^{\frac{-\tau}{m-1}} D_T^{\frac{-1}{m-1}} |\partial_{tt} N_T|^{\frac{m}{m-1}} dxdt =
\left\{
\begin{array}{lll}
O\left(T^{-\frac{(m+1)\theta}{m-1}}\ln T\right) &\mbox{if }\; \tau \geq N(m-1),\\

O\left(T^{N -\frac{\tau + (m+1)\theta}{m-1}}\right) &\mbox{if }\; \tau < N(m-1).
\end{array}
\right.
$$
\end{lem}

\begin{proof}
Without loss of generality, consider $\O = B$. By the definitions of $D_T$ and $N_T$, thanks to Lemma \ref{LL10}, we get
\begin{align*}
\begin{split}
&\int_{Q} |x|^{\frac{-\tau}{m-1}} D_T^{\frac{-1}{m-1}} |\partial_{tt} N_T|^{\frac{m}{m-1}}\,dx\,dt\\
= &\;
\int_0^\infty \vartheta_T(t)^{\frac{-1}{m-1}} |\vartheta_T''(t)|^{\frac{m}{m-1}} dt \int_{B^c} |x|^{\frac{-\tau}{m-1}} H(x)^{-\frac{1}{m-1}} \xi^k\left(\frac{x}{T}\right) dx\\
\leq & \; CT^{-\frac{(m+1)\theta}{m-1}}\int_{1< |x| < 2T} |x|^{\frac{-\tau}{m-1}} H(x)^{-\frac{1}{m-1}} dx.
\end{split}
\end{align*}
Applying Lemmas \ref{LL1}--\ref{LL3} with
$\alpha=\frac{-\tau}{m-1}$ and $\beta =\frac{-1}{m-1} \in (-1, 0)$ (here $m > 2$ was used), we obtain the desired estimate.
\end{proof}

Similarly, we have
\begin{lem}\label{LL23}
Let $\tau\in \mathbb{R}$, $\theta>0$, $m>1$, $k>\frac{2m}{m-1}$ and $N \geq 2$. Then
$$
\int_{Q} |x|^{\frac{-\tau}{m-1}} D_T^{\frac{-1}{m-1}} |\Delta N_T|^{\frac{m}{m-1}}\,dx\,dt= O\left(T^{N -2 + \theta - \frac{\tau+2}{m-1}}\right),\quad \mbox{as }T\to +\infty.
$$
\end{lem}

\begin{proof}
Using \eqref{newExi}, there holds, for large $T$,
\begin{align*}
\begin{split}
&\int_{Q} |x|^{\frac{-\tau}{m-1}} D_T^\frac{-1}{m-1} |\Delta N_T|^\frac{m}{m-1} dxdt\\
= & \; \int_0^\infty \vartheta_T(t) dt \times
\int_{\O^c} |x|^\frac{-\tau}{m-1} H(x)^\frac{-1}{m-1} \xi\left(\frac{x}{T}\right)^\frac{-k}{m-1} \left|\Delta\left[\xi^k\left(\frac{x}{T}\right)\right]\right|^\frac{m}{m-1} dx\\
\leq & \; CT^{\theta - \frac{2m}{m-1}}\int_{T< |x| < 2T} |x|^\frac{-\tau}{m-1} H(x)^\frac{-1}{m-1} \xi\left(\frac{x}{T}\right)^{k-\frac{2m}{m-1}}dx\\
\leq & \; CT^{\theta - \frac{2m}{m-1}}\int_{T< |x| < 2T} |x|^\frac{-\tau}{m-1} H(x)^\frac{-1}{m-1}dx\\
\leq & \; CT^{\theta - \frac{2m}{m-1}}\int_{T< |x| < 2T} |x|^\frac{-\tau}{m-1}dx\\
= &\; CT^{N -2 + \theta - \frac{\tau+2}{m-1}}.
\end{split}
\end{align*}
So we are done.
\end{proof}

\section{Two dimensional situation}\label{sec3}
In this section, we prove successively the parts ${\rm (i)}$, ${\rm (ii)}$ and ${\rm (iii)}$ of Theorem \ref{T1} for $N = 2$. We will detail the proof for ${\rm (i)}$. The proofs for parts ${\rm (ii)}$ and ${\rm (iii)}$ are similar, so we proceed more quickly. Let $p,q>1$ and fix
\begin{equation}\label{ask}
k>\max\left\{\frac{2p}{p-1},\frac{2q}{q-1}\right\}.
\end{equation}
As mentioned above, we consider only $\O = B$, and we explain in Remarks \ref{rnew1}--\ref{rnew2} how the same ideas work for general case.

\subsection{Proof of part (i)}
We argue by contradiction by assuming that the pair $(u,v)\in L^q_{loc}(Q)\times L^p_{loc}(Q)$ is a global weak solution to \eqref{P}--\eqref{BC1}.  For $T> 1$ and $\theta>0$, taking $\varphi=D_T$ in \eqref{wks11}, then
$$
\int_{Q} |x|^a |v|^p D_T  dxdt-\int_{\Gamma}\frac{\partial D_T}{\partial \nu} f\,d\sigma\,dt
\leq
\int_{Q}|u| |\Box D_T| dxdt.
$$
Moreover, as $\p_\nu H$ is constant on $\p B$,
\begin{align}
\label{intpO}
\begin{split}-\int_{\Gamma}\frac{\partial D_T}{\partial \nu} f d\sigma dt&=
C \int_0^\infty \vartheta\left(\frac{s}{T^\theta}\right)^k ds\times \int_{\partial B} f(x) d\sigma = CI_fT^\theta,
\end{split}
\end{align}
where $C$ is a constant depending only on $H$ and $\vartheta$. This yields
\begin{equation}\label{wkss11}
\int_{Q} |x|^a |v|^p D_T \,dxdt+ I_f T^\theta
\leq C \int_{Q}|u| |\Box D_T| dxdt.
\end{equation}
Similarly, taking $\varphi=D_T$ in \eqref{wks12}, we get
\begin{equation}\label{wkss12}
\int_{Q} |x|^b |u|^q D_T dxdt +I_g T^\theta \leq C \int_{Q}|v| |\Box D_T| dxdt.
\end{equation}
By H\"older's inequality, there holds
\begin{equation}\label{ESST1}
\int_{Q} |u| |\partial_{tt}D_T| dxdt\leq \left(\int_{Q} |x|^b |u|^q D_T dxdt\right)^{\frac{1}{q}} \left(\int_{Q} |x|^{\frac{-b}{q-1}} D_T^{\frac{-1}{q-1}}  |\partial_{tt}D_T|^{\frac{q}{q-1}} dxdt\right)^{\frac{q-1}{q}}.
\end{equation}
Using Lemma \ref{LL11} with $\tau=b$ and $m=q$, we obtain
\begin{equation}\label{ESST2}
\int_{Q} |x|^{\frac{-b}{q-1}} D_T^{\frac{-1}{q-1}} |\partial_{tt} D_T|^{\frac{q}{q-1}} dxdt = \left\{
\begin{array}{lll}
O\left(T^{2-\frac{b+(q+1)\theta}{q-1}} \ln T\right) &\mbox{if }\; b<2(q-1),\\
O\left(T^{-\frac{(q+1)\theta}{q-1}}(\ln T)^{2}\right) &\mbox{if }\; b=2(q-1),\\
O\left(T^{-\frac{(q+1)\theta}{q-1}}\right) &\mbox{if }\; b> 2(q-1),
\end{array}
\right.
\end{equation}
as $T\to +\infty$.

\medskip
On the other hand,
\begin{equation}\label{ESST3}
\int_{Q} |u| |\Delta D_T| dxdt\leq \left(\int_{Q} |x|^b |u|^q D_T dxdt\right)^{\frac{1}{q}} \left(\int_{Q} |x|^{\frac{-b}{q-1}} D_T^{\frac{-1}{q-1}}  |\Delta D_T|^{\frac{q}{q-1}} dxdt\right)^{\frac{q-1}{q}}.
\end{equation}
Applying Lemma \ref{LL18} with $\tau=b$ and $m=q$, we have
\begin{equation}\label{ESST4}
\left(\int_{Q} |x|^{\frac{-b}{q-1}} D_T^{\frac{-1}{q-1}}  |\Delta D_T|^{\frac{q}{q-1}} dxdt\right)^{\frac{q-1}{q}}=
O\left(T^\frac{\theta(q-1) -b -2}{q}(\ln T)^{\frac{q-1}{q}}\right),\quad \mbox{as }T\to +\infty.
\end{equation}

Combining \eqref{wkss11} with \eqref{ESST1}--\eqref{ESST4}, for $T$ large enough, there holds
\begin{equation}\label{CC1}
J_T(a,v)+ I_f T^\theta \leq  C [J_T(b,u)]^{\frac{1}{q}} \alpha(T),
\end{equation}
where
$$
J_T(a,v)=\int_{Q} |x|^a |v|^p D_T  dxdt, \quad J_T(b,u)=\int_{Q} |x|^b |u|^q D_T dxdt
$$
and
\begin{equation}\label{alphaT}
\alpha(T)= T^\frac{\theta(q-1) -b -2}{q}(\ln T)^{\frac{q-1}{q}}+
\left\{
\begin{array}{lll}
T^{\frac{-(q+1)\theta}{q}}(\ln T)^{\frac{2q-2}{q}} &\mbox{if }\; b\geq 2(q-1),\\
T^\frac{2(q-1) -b - (q+1)\theta}{q} (\ln T)^{\frac{q-1}{q}}&\mbox{if }\; b<2(q-1).
\end{array}
\right.
\end{equation}

Exchanging now the roles of $u$ and $v$, using \eqref{wkss12}, we have also
\begin{equation}\label{CC2}
J_T(b,u)+ I_g T^\theta \leq  C [J_T(a,v)]^{\frac{1}{p}} \beta(T),
\end{equation}
where
\begin{equation}\label{betaT}
\beta(T)= T^\frac{\theta(p-1) -a -2}{p}(\ln T)^{\frac{p-1}{p}}+
\left\{
\begin{array}{lll}
T^{\frac{-(p+1)\theta}{p}}(\ln T)^{\frac{2p-2}{p}} &\mbox{if }\; a\geq 2(p-1),\\
T^\frac{2(p-1) -a - (p+1)\theta}{p} (\ln T)^{\frac{p-1}{p}}&\mbox{if }\; a<2(p-1).
\end{array}
\right.
\end{equation}

Without loss of generality, we assume $I_f > 0$, as $(I_f, I_g) \succ (0, 0)$. Combining \eqref{CC1} and \eqref{CC2}, there holds, for large $T$,
$$
J_T(a,v)+ T^\theta \leq  CJ_T(a,v)^{\frac{1}{pq}} \beta(T)^{\frac{1}{q}} \alpha(T).
$$
Using Young's inequality, we get
\begin{equation}\label{FEQ}
T^{-\theta} \alpha(T)^{\frac{pq}{pq-1}}\beta(T)^{\frac{p}{pq-1}} \geq C > 0, \quad \mbox{for large $T$}.
\end{equation}
However, we claim that with large $\theta>0$,
\begin{equation}\label{claimL}
\lim_{T\to +\infty} T^{-\theta} \alpha(T)^{\frac{pq}{pq-1}} \beta(T)^{\frac{p}{pq-1}}=0.
\end{equation}

\medskip
By \eqref{alphaT} and \eqref{betaT}, for $\theta > 0$ large enough, there hold
\begin{align}
\label{AB}
\alpha(T) \sim T^\frac{\theta(q-1) -b -2}{q}(\ln T)^{\frac{q-1}{q}}, \;\; \beta(T) \sim T^\frac{\theta(p-1) -a -2}{p}(\ln T)^{\frac{p-1}{p}}, \quad \mbox{as } \; T \to +\infty.
\end{align}
Therefore
\begin{align}
\label{TAB}
T^{-\theta} \alpha(T)^{\frac{pq}{pq-1}}\beta(T)^{\frac{p}{pq-1}} \sim T^{-\frac{(b+2)p+(a+2)}{pq-1}}\ln T, \quad \mbox{as } \; T \to +\infty,
\end{align}
hence \eqref{claimL} holds true (with large but fixed $\theta$) since $(a, b) \succ (-2, -2)$. Obviously, \eqref{claimL} is not compatible with \eqref{FEQ}, which means that no global weak solution exists. This proves part (i) of Theorem \ref{T1} for $N = 2$.

\begin{rem}
\label{rnew1}
For general smooth open sets $\O$, we have no longer $\p_\nu H_\O \equiv$ constant on $\p\O$, hence we have no longer the equality \eqref{intpO} for all $f \in L^1(\p\O)$. However, by Hopf's Lemma, $\p_\nu H_\O \leq -C_\O < 0$ on $\p\O$. If now $f \geq 0$ and $T > {\rm dist}(0, \p\O)$, there holds
\begin{align*}
-\int_{\Gamma}\frac{\partial D_T}{\partial \nu} f d\sigma dt& \geq
C_\O \int_0^\infty \vartheta\left(\frac{s}{T^\theta}\right)^k ds\times \int_{\partial\O} f(x) d\sigma \geq CI_fT^\theta.
\end{align*}
where $C$ depends only on $\O$ and $\vartheta$. It's easy to see that all the arguments are still valid for $f, g \geq 0$.
\end{rem}

\subsection{Proof of part (ii)}

Assume that $(u,v)\in L^q_{loc}(Q)\times L^p_{loc}(Q)$ is a global weak solution to \eqref{P}--\eqref{BC2}. Let
$$
K_T(a,v)=\int_{Q} |x|^a |v|^p N_T dxdt,\quad K_T(b,u)=\int_{Q} |x|^b |u|^q N_T dxdt.
$$
By H\"older's inequality,
\begin{align}
\label{ESSTN1}
\begin{split}
\int_{Q} |u| |\Box N_T| dxdt \leq C K_T(b, u)^{\frac{1}{q}} \left(\int_{Q} |x|^{\frac{-b}{q-1}} N_T^{\frac{-1}{q-1}}|\Box N_T|^{\frac{q}{q-1}} dxdt\right)^{\frac{q-1}{q}}.
\end{split}
\end{align}
Applying Lemmas \ref{LL13}--\ref{LL20} with $\tau=b$, $m=q$ and $N=2$, remarking that the involved estimates are exactly of the same order or better than those in Lemmas \ref{LL11} and \ref{LL18}, we deduce that for $T$ large,
\begin{align}
\label{ESSTN2}
\left(\int_{Q} |x|^{\frac{-b}{q-1}} N_T^{\frac{-1}{q-1}} |\Box N_T|^{\frac{q}{q-1}} dxdt\right)^{\frac{q-1}{q}} \leq C\alpha(T)
\end{align}
where $\alpha(T)$ is given by \eqref{alphaT}. Similarly, there holds, for $T$ large,
\begin{align}\label{ESSTN5}
\begin{split}
\int_{Q} |v| |\Box N_T| dxdt &\leq C K_T(a, v)^\frac{1}{p} \left(\int_{Q} |x|^{\frac{-a}{p-1}} N_T^{\frac{-1}{p-1}}  |\Box N_T|^{\frac{p}{p-1}} dxdt\right)^{\frac{p-1}{p}}\\
 & \leq C K_T(a, v)^\frac{1}{p}\beta(T),
\end{split}
\end{align}
where $\beta(T)$ is given by \eqref{betaT}. Moreover, by the definition of $N_T$, for $T$ large,
\begin{align}
\label{intpN}
\int_{\Gamma} fN_T d\sigma dt = \int_{\Gamma} f \vartheta_T(t) d\sigma dt = CI_f T^\theta, \quad \int_{\Gamma} gN_T d\sigma dt = CI_g T^\theta.
\end{align}
Take $\psi=N_T$ in \eqref{wks21}--\eqref{wks22},  combining with \eqref{ESSTN1}--\eqref{ESSTN5}, we get
\begin{equation}\label{CCN2}
K_T(a,v)+ I_f T^\theta \leq C K_T(b,u)^\frac{1}{q}\alpha(T), \quad K_T(b,u)+ I_g T^\theta \leq C K_T(a,v)^{\frac{1}{p}} \beta(T).
\end{equation}

Remark that \eqref{CCN2} is just \eqref{CC1} and \eqref{CC2}, if we replace $K_T$ by $J_T$. Assuming without loss of generality $I_f > 0$, repeating the previous arguments for part ${\rm (i)}$, \eqref{FEQ} still holds true. However, we can always choose $\theta>0$ large to get \eqref{claimL}, which makes \eqref{FEQ} impossible. We reach a contradiction.

\begin{rem}
\label{rnew2}
To get \eqref{intpN}, we used only $\xi\left(\frac{x}{T}\right) \equiv 1$ on $\p\O$ when $T$ is large, which is true for general bounded open sets $\O$.
\end{rem}

\subsection{Proof of part (iii)}
We use again the method of contradiction. Assume that $(u,v)\in L^q_{loc}(Q)\times L^p_{loc}(Q)$ is a global weak solution to \eqref{P}--\eqref{BC3}, with now $p>2$. We take $(D_T, N_T)$ as a couple of test functions, and use the same notations $J_T$, $K_T$, $\alpha(T)$ and $\beta(T)$ as before.

\medskip
Inserting $\varphi=D_T$ in \eqref{wks21}, we obtain, for $T$ large,
\begin{equation}\label{CCN3}
J_T(a,v)+ I_f T^\theta \leq C[J_T(b,u)]^\frac{1}{q}\alpha(T) \leq C[K_T(b,u)\ln T]^\frac{1}{q}\alpha(T).
\end{equation}
The key point here is to estimate $\|v\Box N_T\|_{L^1(Q)}$ using $J_T(a, v)$. By H\"older's inequality,
\begin{equation}\label{ESSTND5}
\begin{split}
\int_{Q} |v| |\Box N_T| dxdt & \leq J_T(a,v)^{\frac{1}{p}} \left(\int_{Q} |x|^{\frac{-a}{p-1}} D_T^{\frac{-1}{p-1}} |\Box N_T|^{\frac{p}{p-1}}\,dx\,dt\right)^{\frac{p-1}{p}}\\
& \leq C J_T(a,v)^{\frac{1}{p}}\beta(T).
\end{split}
\end{equation}
The last inequality follows from Lemmas \ref{LL16}--\ref{LL23} with $\tau = a$, $m = p$ and $N = 2$. Moreover, let $\psi = N_T$ in \eqref{wks22}, using \eqref{ESSTND5}, there holds
\begin{equation}\label{AF2}
K_T(b,u)+I_g T^\theta \leq C \int_{Q} |v||\Box N_T| dxdt \leq C J_T(a,v)^{\frac{1}{p}}\beta(T).
\end{equation}

$\bullet$ Assume first $I_f > 0$, combining \eqref{CCN3} and \eqref{AF2}, we deduce that
$$
J_T(a,v)+ T^\theta \leq CJ_T(a,v)^\frac{1}{pq}\alpha(T)[\beta(T)\ln T]^\frac{1}{q}.
$$
Applying Young's inequality, there holds
\begin{equation*}
T^{-\theta} \alpha(T)^{\frac{pq}{pq-1}} [\beta(T)\ln T]^{\frac{p}{pq-1}} \geq C > 0, \quad \mbox{for large $T$}.
\end{equation*}
However, fix $\theta > 0$ large, we have still \eqref{TAB}, which is impossible seeing the above estimate.

\smallskip
$\bullet$ Assume now $I_g > 0$. Always using \eqref{CCN3} and \eqref{AF2}, there holds
$$
K_T(b, u)+ T^\theta \leq CK_T(b,u)^\frac{1}{pq}\beta(T)[\alpha(T)(\ln T)^{\frac{1}{q}}]^\frac{1}{p},
$$
hence
\begin{equation*}
T^{-\theta} [\alpha(T)\ln T]^{\frac{q}{pq-1}}\beta(T)^{\frac{pq}{pq-1}} \geq C > 0, \quad \mbox{for large $T$}.
\end{equation*}
Moreover, fixing a large $\theta$ such that \eqref{AB} is valid, we get, as $T\to +\infty$,
$$T^{-\theta} [\alpha(T)\ln T]^{\frac{q}{pq-1}}\beta(T)^{\frac{pq}{pq-1}} \sim T^{-\frac{(b+2)+(a+2)q}{pq-1}}(\ln T)^{1+\frac{q}{pq-1}}.$$
This contradicts the previous inequality.

\medskip
To conclude, if $(I_f, I_g) \succ (0, 0)$ and $(a, b) \succ (-2, -2)$, there exists always a contradiction if a global weak solution exists. The proof of part (iii) is completed for $N = 2$. \qed

\section{Proof of Theorem \ref{T1} for $N \geq 3$}\label{sec4}

Let $N \geq 3$, $p,q>1$ and $k$ satisfy \eqref{ask}. As above, we can consider just $\Omega = B$. The proof is very similar to the case $N = 2$.

\subsection{Proof of parts (i)--(ii)} Without restriction of the generality, suppose $I_f>0$ and
\begin{equation}\label{aspqN}
\delta + 2 = \frac{2p(q+1)+pb+a}{pq-1}>N,
\end{equation}
where $\delta$ is defined by \eqref{dg}. Assume that $(u,v)\in L^q_{loc}(Q)\times L^p_{loc}(Q)$ is a global weak solution to \eqref{P}--\eqref{BC1}. Proceeding as above, by Lemmas \ref{LL12} and \ref{LL19} with $\tau=b$ and $m=q$, H\"older and Young's inequalities,  we obtain again \eqref{FEQ} with now
\begin{equation}\label{alphaTN3}
\alpha(T)=T^{\frac{(N-2+\theta)(q-1)-b-2}{q}}+
\left\{
\begin{array}{lll}
T^{\frac{-(q+1)\theta}{q}}(\ln T)^{\frac{q-1}{q}} &\mbox{if}& b\geq N(q-1),\\
T^{\frac{N(q-1)-b-(q+1)\theta}{q}} &\mbox{if}& b<N(q-1)
\end{array}
\right.
\end{equation}
and
\begin{equation}\label{betaTN3}
\beta(T)=T^{\frac{(N-2+\theta)(p-1)-a-2}{p}}+
\left\{
\begin{array}{lll}
T^{\frac{-(p+1)\theta}{p}}(\ln T)^{\frac{p-1}{p}} &\mbox{if}& a\geq N(p-1),\\
T^{\frac{N(p-1)-a-(p+1)\theta}{p}} &\mbox{if}& a<N(p-1).
\end{array}
\right.
\end{equation}

Taking $\theta$ large enough, when $T\to +\infty$, there holds
\begin{align}
\label{ABsim}
\alpha(T)\sim T^{\frac{(N-2+\theta)(q-1)-b-2}{q}} \quad \mbox{and} \quad \beta(T)\sim T^{\frac{(N-2+\theta)(p-1)-a-2}{p}}.
\end{align}
Hence
$$
T^{-\theta} \alpha(T)^{\frac{pq}{pq-1}}\beta(T)^{\frac{p}{pq-1}}\sim
T^{N - 2 -\frac{(b+2)p + (a+2)}{pq-1}} = T^{N - 2 - \delta}.
$$
Thanks to \eqref{aspqN}, \eqref{claimL} follows by choosing a large $\theta$.

\medskip
The contradiction between \eqref{FEQ} and \eqref{claimL}  means that no global weak solution exists for \eqref{P}--\eqref{BC1}. The nonexistence result for \eqref{P}--\eqref{BC2} can be derived by similar arguments, so we omit the proof.

\subsection{Proof of part (iii)}
Let $p>2$ and suppose that $(u,v)\in L^q_{loc}(Q)\times L^p_{loc}(Q)$ is a global weak solution to \eqref{P}--\eqref{BC3}. For $T>1$, using $\varphi=D_T$ in \eqref{wks21}, we can claim that
\begin{equation}\label{CCN3couf}
J_T(a,v)+ I_f T^\theta \leq C[J_T(b,u)]^\frac{1}{q}\alpha(T) \leq C[K_T(b,u)]^\frac{1}{q}\alpha(T).
\end{equation}
Here we used $H(x) \leq 1$ as $N \geq 3$.

\medskip
Proceeding as in the proof of part (iii) for $N = 2$, taking $\psi = N_T$ in \eqref{wks22}, we get \eqref{AF2}. Here $\alpha(T)$ and $\beta(T)$ are given by {\eqref{alphaTN3}
and \eqref{betaTN3}. Assume first $I_f > 0$ and \eqref{aspqN} holds. Using \eqref{CCN3couf} and \eqref{AF2}, we have still \eqref{FEQ}, but also the claim \eqref{claimL} for $\theta$ large enough, which is impossible.

\medskip
Assume now $I_g > 0$ and
$$
\gamma + 2 = \frac{2q(p+1)+qa+b}{pq-1}>N,
$$
with $\gamma$ given by \eqref{dg}. Combining \eqref{CCN3couf} and \eqref{AF2}, there holds
$$
K_T(b, u)+ T^\theta \leq CK_T(b,u)^\frac{1}{pq}\beta(T)\alpha(T)^\frac{1}{p},
$$
hence
$$
T^{-\theta} \alpha(T)^{\frac{q}{pq-1}}\beta(T)^{\frac{pq}{pq-1}} \geq C > 0, \quad \mbox{as }\; T\to \infty.
$$
We can conclude if
\begin{align}
\label{new4.6}
\lim_{T\to +\infty} T^{-\theta} \alpha(T)^{\frac{q}{pq-1}}\beta(T)^{\frac{pq}{pq-1}}=0.
\end{align}
Taking $\theta$ large enough, by \eqref{ABsim}, there holds $T^{-\theta}\alpha(T)^{\frac{q}{pq-1}}\beta(T)^{\frac{pq}{pq-1}} \sim T^{N -2 - \gamma}$ for $T$ large, so \eqref{new4.6} holds true and the proof of part (iii) is completed. \qed

\section{Further Remarks}\label{sec5}
It's worthy to mention that the system of wave equations in the whole space, i.e.
$$\Box u = |v|^p, \;\; \Box v = |u|^q \quad \mbox{in }\; (0,\infty)\times \R^N, \quad p, q > 1, \; N \geq 2$$
has been extensively studied since the seminal work \cite{DGM}. It is showed that for compactly supported initial data with positive averages for $\p_t u(0, x)$, $\p_t v(0, x)$, there exists a critical curve for the global existence, which is
$$\max\left\{\frac{p+2+q^{-1}}{pq-1}, \frac{q+2+p^{-1}}{pq-1}\right\} = \frac{N-1}{2}.$$
The corresponding  system of inequalities was studied in \cite{PV}, where Theorem 6 (see also Application 2) proves the nonexistence of nontrivial global solution if
$$1 < p, q < \frac{N+1}{N-1}, \;\; \int_{\R^N} \p_t u(0, x)dx \geq 0， \;\; \int_{\R^N} \p_t v(0, x)dx \geq 0.$$
We can see that the critical criteria in the above cases are quite different for our situation. This phenomenon is similar to comparing Strauss's critical exponent $p_c(N)$ for \eqref{PRN}, Kato's exponent $\widetilde p_c(N)$ for \eqref{PERN} and Zhang's exponent $p^*$ for \eqref{PRNE}. In other words, the blow-up for inequalities on exterior domains is of very different nature comparing to the whole space situation.

\medskip
The critical case $N \geq 3$,
$$
\max\left\{{\rm sgn}(I_f)\times\frac{2p(q+1)+pb+a}{pq-1},\;\; {\rm sgn}(I_g)\times \frac{2q(p+1)+qa+b}{pq-1}\right\}= N
$$
for the system \eqref{P} is not investigated here. It should be interesting to decide whether this critical curve in $(p, q)$--plan belongs to the blow-up situation.

\medskip
For the mixed boundary condition case \eqref{BC3}, we supposed that $p>2$ due to technical reason. It should be interesting to consider the case $1<p\leq 2$.

\medskip
As indicated in Remark \ref{signfg}, the case of wave inequalities, under homogeneous constraints, i.e. $f = g = 0$, is very special. We may have no critical criteria of Fujita type in general. However, the simple example there only works for $a, b \leq 0$. It could be interesting to understand the long term behavior of solutions to \eqref{P} with $a, b > 0$ and various type of homogeneous constraints with $f=g=0$. 

\medskip
In the case of homogeneous constraints, another way to avoid the simple example in Remark \ref{signfg} is to add sign condition or nonnegative average constraint on $\p_t u(0, x)$, $\p_t v(0, x)$ as in \cite{K, PV}. For example, consider the following problem:
$$
\Box u \geq |x|^a |u|^p \; \mbox{ in }\; \R_+\times B_r^c, \quad u \geq 0\; \mbox{ on }\;  \R_+\times \p B_r\quad \mbox{and}\quad \partial_t u(0, x) \geq 0,
$$
where $B_r \subset \R^N$, $N \geq 3$, $a > -2$. Laptev \cite{L} showed that the critical exponent for existence of non trivial global solution is $\frac{N+1 + a}{N - 1}$.

\medskip
The understanding for wave equation on exterior domains with homogeneous Dirichlet boundary condition is more difficult. Consider \eqref{PRNE} with $N \geq 2$, $a = 0$ and $u = 0$ on $\p\O$. There are many works who suggest that the critical exponent of Fujita type could be the same as for the whole space, i.e. $p_c(N)$ given by \eqref{PCN}.
\begin{itemize}
\item Let $1 < p \leq p_c(N)$, it's showed that for special choice of $(u_0, u_1) \succ (0, 0)$, the solution to \eqref{PRNE} with $(u, \p_t u)|_{t=0} = (\e u_0, \e u_1)$ will blow up for any $\e > 0$, see \cite{LZ} and the references therein. However, the blow-up result for general $(u_0, u_1)$ seems unknown.
\item For $p > p_c(N)$, there exist some global existence results for some $p > p_c(N)$ in low dimensions $N \leq 4$ with non trapping obstacle $\O$ and suitable $u_0, u_1 > 0$. See for instance \cite{D, SSW}.
\end{itemize}
As far as we know, it seems that there is no general result for the global existence of wave equation on exterior domains \eqref{PRNE} with homogeneous Neumann boundary condition.

\bigskip\noindent
{\bf Acknowledgments}. M.J.~ and B.S.~ extend their appreciation to the Deanship of Scientific Research at King Saud University for funding this work through research group No. RGP-237. D.Y.~is partially supported by Science and Technology Commission of Shanghai Municipality (STCSM), grant No. 18dz2271000.

\end{document}